\documentclass[12pt,a4paper]{article}
\usepackage[utf8]{inputenc}
\usepackage{amsmath}
\usepackage{amsfonts}
\usepackage{amssymb}
\usepackage{amsthm}
\usepackage{graphicx}
\usepackage{kpfonts}
\usepackage{hyperref}
\usepackage[width=16.00cm, height=24.00cm]{geometry}
\usepackage{enumitem}
\usepackage{epigraph}

\newtheorem{theorem}{Theorem}[section]
\newtheorem{lemma}[theorem]{Lemma}
\newtheorem{proposition}[theorem]{Proposition}

\theoremstyle{definition}

\newtheorem{example}[theorem]{Example}
\newtheorem{remark}[theorem]{Remark}
\newtheorem{question}[theorem]{Question}

\newcommand{\mazzi}[1]{\textcolor[RGB]{51,200,126}{\bf #1}}

\DeclareMathOperator{\card}{card}

\DeclareMathOperator{\re}{Re}
\DeclareMathOperator{\im}{Im}

\title{A note on the symmetry of sequence spaces}

\author{
Daniel Carando \thanks{Departamento de Matem\'atica, Facultad de Cs. Exactas y Naturales,
		Universidad de Buenos Aires and IMAS-UBA-CONICET, Int. G\"uiraldes s/n, 1428, Buenos Aires, Argentina (dcarando@dm.uba.ar). Supported by CONICET-PIP 11220130100329CO and ANPCyT PICT 2015-2299.}
\and
Martín Mazzitelli \thanks{Universidad Nacional del Comahue, CONICET, Departamento de Matemática, Facultad de Economía y Administración, Neuquén, Argentina (martin.mazzitelli@crub.uncoma.edu.ar). Supported by CONICET-PIP 11220130100329CO and ANPCyT PICT 2015-2299.}
\and
Pablo Sevilla-Peris\thanks{Institut Universitari de Matem\`atica Pura i Aplicada,
Universitat Polit\`{e}cnica de Val\`encia, cmno Vera s/n, 46022,
Val\`encia, Spain (psevilla@mat.upv.es) Supported by MINECO and FEDER project MTM2017-83262-C2-1-P}}
 
 \date{}

\begin{document}
\maketitle

\begin{abstract}
We give a self-contained treatment of symmetric Banach sequence spaces and some of their natural properties. We are particularly interested in the symmetry of the norm and the existence of symmetric linear functionals. Many of the presented results are known or commonly accepted but are not found in the literature.
\end{abstract}

\section{Introduction}

\epigraph{ Reality  likes  symmetry  and  faint  anachronisms.}{\textit{The South} \\ \textsc{J. L. Borges}}

In this expository work, we study some natural properties of symmetric Banach sequence spaces. We aim to give a simple and self-contained presentation of different results which are probably known but hard to find in the literature. One of our goals is to prove  that the two natural definitions for symmetry (for a Banach sequence space together with its norm) are actually equivalent (see Theorem~\ref{simequiv} below). We are also naturally led to  investigate the existence of symmetric linear functionals on these spaces.  
The structure of the article prioritizes the exposition and the development of ideas over briefness. Thus, we take our time to discuss different reasonings that may not lead to the desired results, but either give a hint about the correct path  or raise natural related questions which are interesting on their own.

\bigskip
Given $\mathbb{K}=\mathbb{R}$ or $\mathbb{C}$ we consider $\mathbb{K}^{\mathbb{N}}$, the vector space of all sequences $(x_n)_{n\in\mathbb{N}}$ with $x_n\in \mathbb{K}$. A  Banach sequence space is a subspace $X$ of $\mathbb{K}^{\mathbb{N}}$ endowed with a complete norm such that, if $x=(x_{n})_{n} \in \mathbb{K}^{\mathbb{N}}$ and $y=(y_{n})_{n} \in X$ are sequences so that $\vert x_{n} \vert \leq \vert y_{n} \vert$ for every $n$, then $x \in X$ and $\Vert x \Vert \leq \Vert y \Vert$.

A Banach sequence space is symmetric if $x_{\sigma} = (x_{\sigma(n)})_{n} \in X$ and $\Vert x_{\sigma} \Vert = \Vert x \Vert$ for every $x \in X$ and all permutation $\sigma$ of $\mathbb{N}$.
Let us note that if $X$ is symmetric then $\Vert e_{k} \Vert = \Vert e_{1} \Vert$ for all $k$ and, multiplying the norm by a number if necessary, we may assume
\begin{equation} \label{norma1}
\Vert e_{k} \Vert   = 1, \quad \text{ for all  } k\in\mathbb{N}.
\end{equation}
From now on unless stated the contrary, all Banach sequence spaces (symmetric or not) are assumed to satisfy \eqref{norma1}. Then a straightforward computation (left to the reader) shows that
every $X$ Banach sequence space satisfies
\begin{equation} \label{continuous inclusions}
\ell_{1} \hookrightarrow X \hookrightarrow  \ell_{\infty}
\end{equation}
with continuous inclusions of norm one. As a consequence,  convergence in norm implies coordinate-wise convergence.\\

Given a bounded sequence $x=(x_{n})_{n}$ its decreasing rearrangement $x^{*}=(x^*_n)_{n\in \mathbb{N}}$ is
defined by
\[
x^{*}_{n} := \inf \{ \sup_{k \in \mathbb{N} \setminus J } \vert x_{k} \vert  \colon  J \subset \mathbb{N} \  , \ \card{ J } < n \}, \quad \text{for each $n \in \mathbb{N}$.}
\]

Symmetric spaces are also usually defined as those satisfying that $x \in X$ if and only if $x^{*} \in X$ and, in this case $\Vert x \Vert = \Vert x^{*} \Vert$. Both definitions are commonly accepted as equivalent, but we could not find any complete proof of this fact in the literature, and the partial results are scattered on different publications. Our aim here is to provide with a unified, complete and self-contained proof of this fact, complemented with some other independently interesting questions that came across. So our goal is to prove the following theorem.
\begin{theorem}  \label{simequiv}
Let $X$ be a Banach sequence space. The following statements are equivalent.
\begin{enumerate}
\item \label{simequiv 1} $X$ is symmetric.
\item \label{simequiv 2} $x \in X$ if and only if $x^{*} \in X$ and, in this case, $\Vert x \Vert = \Vert x^{*} \Vert$.
\end{enumerate}
\end{theorem}

Before we proceed to the proof (that we give in Section~\ref{proof thm symmetric}) let us make a couple of comments. First of all, if $\sigma$ is a permutation of $\mathbb{N}$, one clearly has $x^{*} = (x_{\sigma})^{*}$ for every $x$ and, therefore,
\ref{simequiv 2} implies \ref{simequiv 1} in the previous theorem. The proof only needs to focus on the reverse implication.\\

Garling was the first one to consider symmetric sequence spaces, with a different point of view from the one we take here. Let us briefly explain it. For him a sequence space is a subspace $X$ of $\omega$
($\mathbb{K}^{\mathbb{N}}$ endowed with the topology of coordinatewise convergence, which makes it a Fr\'echet space) satisfying that if $x \in \omega$ and $y \in X$ are so that $\vert x_{n} \vert \leq \vert y_{n} \vert$
for every $n$, then $x \in X$. It is symmetric whenever $x_{\sigma} \in X$ for every $x \in X$ and every permutation $\sigma$ (note that, since he was working with locally convex spaces no reference to a norm is made). He showed in
\cite[Proposition~6]{Gar} (see also Proposition~\ref{aquila}) that if $X$ is symmetric, then either $X \subseteq c_{0}$, $X=\ell_{\infty}$ or $X=\omega$ and in \cite[Proposition~7]{Gar} that, if $X \subseteq c_{0}$ then
$x^{*} \in X$ if and only if $x \in X$. Since this is obviously satisfied by $\ell_{\infty}$ and $\omega$, this gives a sort of non-normed version of Theorem~\ref{simequiv}. Let us see now how do the norms come into play.\\

If $X$ is a symmetric Banach sequence space, a  mapping $\gamma\colon X\to \mathbb{K}$ is said to be symmetric if $\gamma(x)=\gamma(x_\sigma)$ for every permutation $\sigma$ of $\mathbb{N}$. Then, a Banach sequence space
is symmetric (in `our' sense) if and only if it is symmetric (in the sense of Garling) and moreover the norm defining the topology is symmetric. The following example shows that being symmetric in the sense of Garling does not necessarily imply that the norm is symmetric.
\begin{example}
 Take the sequence of weights $w=(w_n)_n$ given by $w_1=\frac{1}{2}$ and $w_n=1$ for all $n\geq 2$ and let $X=\ell_1(w)=\{x\in \mathbb{K}^{\mathbb{N}} \colon \sum_n x_n w_n <\infty\}$ with the natural norm given by
 \[
 \|x\|_{X}=\sum_{n=1}^\infty |x_n| w_n.
 \]
 Then it is clear that $x\in X$ if and only if $x_\sigma\in X$ for any permutation $\sigma$ of $\mathbb{N}$, but since
$\frac{1}{2}=\|e_1\|_X\neq \|e_n\|_X=1$ for $n\geq 2$, the norm is not symmetric.
 \end{example}

Summarizing, one implication in Theorem~\ref{simequiv} is obvious, and the part of the remaining implication not involving the norm is already covered by Garling's results. So, what is left to complete the proof is to see that if the
norm of $X$ is symmetric norm, then $\|x\|_X=\|x^*\|_{X^*}$ for every $x\in X$.\\

A first attempt could be to try to disprove Theorem~\ref{simequiv} by finding a symmetric Banach sequence space the norm of which does not satisfy the condition. We already knew (see e.g. \cite[Proposition~3.a.3]{LiTz77}) that
this is not possible whenever $X \hookrightarrow c_{0}$. So, taking into account \cite[Proposition~6]{Gar} (or Proposition~\ref{aquila}) the only possibility is to consider $\ell_{\infty}$ renormed in some convenient way.
\begin{remark}
Suppose we could find some symmetric linear functional $\gamma$ on $ \ell_{\infty}$ so that $\gamma (1,1,1, \ldots) =1$. In that case, considering $X=(\ell_\infty, \|\cdot\|_X)$ with
\[
\|x\|_X=\|x\|_\infty+|\gamma(x)|.
\]
This is clearly a symmetric norm. Take $x=(1,0,1,0,\dots)$, so that $x^*=(1,1,1,1,\dots)$.
By the symmetry and linearity of $\gamma$,
\[
2\gamma(1,0,1,0,\dots)=\gamma(1,0,1,0,\dots)+\gamma(0,1,0,1,\dots)=\gamma(1,1,1,1,\dots)=1
\]
wich gives $\gamma(1,0,1,0,\dots)=\frac{1}{2}$. Then,
\[
2=\|(1,1,1,1,\dots)\|_X = \Vert x^{*} \Vert_{X} \neq \Vert x \Vert_{X} = \|(1,0,1,0,\dots)\|_X=1+\frac{1}{2}.
\]
\end{remark}
As Theorem~\ref{simequiv} holds, we deduce form the previous remark that there is no symmetric linear functional $\gamma$ on $\ell_\infty$ satisfying $\gamma(1,1,1,\dots)\neq 0$. But, as a matter of fact, we can prove in a straightforward way that on $\ell_\infty$ there is no symmetric linear functional at all (except of course the zero functional).

\begin{proposition}\label{no symm on linfty}
There is no symmetric linear functional $\gamma \nequiv 0$ on $\ell_{\infty}$.
\end{proposition}
\begin{proof}
Since
$$
(1,0,1,0,1,0,1,\dots)=(1,0,0,0,1,0,0,0,1,\dots)+(0,0,1,0,0,0,1,0,0,0,1,\dots),
$$
by linearity we have
$$
\gamma(1,0,1,0,1,0,1,\dots)=\gamma(1,0,0,0,1,0,0,0,1,\dots)+\gamma(0,0,1,0,0,0,1,0,0,0,1,\dots)
$$
and, by the symmetry of $\gamma$,
$$
\gamma(1,0,1,0,1,0,1,\dots)=\gamma(1,0,0,0,1,0,0,0,1,\dots)=\gamma(0,0,1,0,0,0,1,0,0,0,1,\dots)=0.
$$
Using again the symmetry of $\gamma$ we deduce that $\gamma(x)=0$ for every sequence $x$ having infinitely many $0$'s and infinitely many $1$'s.
Now, every sequence of $0$'s and $1$'s can be written as a sum or difference of two sequences with infinitely many $0$'s and infinitely many $1$'s, so $\gamma$ is zero in all such sequences.
If $x$ belongs to the subset $D\subseteq \ell_\infty$ of all the sequences taking only finite values, then $x$ is a finite linear combination of sequences taking only the values $0$ and $1$. Hence, by linearity of $\gamma$, we have $\gamma(x)=0$. This proves the statement, since $D$ is dense in $\ell_\infty$.
\end{proof}

So we have that on $\ell_{\infty}$ there is no (non-trivial) symmetric functional. On the other side of the scale, $\ell_{1}$, the functional $\gamma((x_n)_n)=\sum_n x_n$ is clearly symmetric. So, we now wonder what happens in between.

\begin{question} \label{pregunta}
Are there symmetric spaces $X\neq \ell_1$ where we can find symmetric linear functionals?
\end{question}

We address this question in Section~\ref{section sym funct}. We will see in Proposition~\ref{symm null on separable} that for separable spaces there is no hope and in Proposition~\ref{symmetric on marcinkiewicz} we
construct a Marcinkiewicz on which there are symmetric non-trivial linear functinonals.

\section{Proof of Theorem~\ref{simequiv} and related results}\label{proof thm symmetric}

We start the way towards the proof of Theorem~\ref{simequiv} by showing that a symmetric Banach sequence space either is contained in $c_{0}$ (with continuous norm) or is $\ell_{\infty}$ (with some
equivalent norm). This, as we already mentioned in the introduction, was essentially proved in  \cite[Proposition~6]{Gar}. There the spaces are not supposed to be normed, hence no reference to
the continuity of the norms is made, but it this follows easily from \eqref{continuous inclusions} (also, a third possibility $X=\omega$ appears there).


\begin{proposition} \label{aquila}
If $X$ is symmetric, then either $X= \ell_{\infty}$ (with equivalent norms) or $X \hookrightarrow  c_{0}$.
\end{proposition}
\begin{proof}
If $X \subseteq c_0$ as sets then, $X\nsubseteq c_0$ follows immediately from \eqref{continuous inclusions}. Suppose, then, that these is some $x=(x_n)_n\in X\setminus c_0$ and choose $\varepsilon >0$ and $J \subseteq \mathbb{N}$ infinite so that $\mathbb{N}\setminus J$  is infinite and $|x_n|>\varepsilon$ for $n \in J$. Define now $y=(y_n)_n$ and  $z=(z_n)_n$ by
\[
 y_n = \left\{\begin{array}
                        [c]{ccc}
                        \varepsilon & \text{if} & n\in J,\\
                        0 & \text{if} & n\in \mathbb{N}\setminus J \,,
                            \end{array}
                    \right.
\quad
z_n = \left\{\begin{array}
                        [c]{ccc}
                        0 & \text{if} & n\in J,\\
                        \varepsilon & \text{if} & n\in \mathbb{N}\setminus J \,.
                            \end{array}
                    \right.
                    \]
Since $|y_n|\leq |x_n|$ for every $n$, we deduce that $y\in X$. On the other hand, the fact that $z=y_\sigma$ for some permutation $\sigma$ and that $X$ is symmetric give that $z \in X$. Hence
\[
(1, 1, 1,\dots)=\frac{y+z}{\varepsilon} \in X.
\]
Now, given any $x\in \ell_\infty$ we have $|x_n|<\|x\|_\infty$ for every $n$ and, since $\|x\|_\infty(1,1,1,\dots) \in X$, we obtain that $x\in X$ with $\|x\|_X\leq \|(1,1,1,\dots)\|_X \|x\|_\infty$. This shows that $\ell_\infty \hookrightarrow X$. From  \eqref{continuous inclusions} we have  $X\hookrightarrow \ell_\infty$ and the proof is completed.
\end{proof}

Given a sequence $x$ in $\mathbb{K}^{\mathbb{N}}$ we introduce some further notation. First of all, we denote $\vert x \vert = (\vert x_{n} \vert)_{n}$. Also, following \cite[Definition~3]{Gar}, the \emph{closing up} $x$ is defined as  the sequence  $x'=(x'_n)_n$ given by
 $$
 x'_n = \left\{\begin{array}
                        [c]{cc}
                        \text{the $n$th non-zero term of $x$} & \text{if such exists,}\\
                        0 & \text{otherwise}.
                            \end{array}
                    \right.
$$
The semigroup of one-to-one mappings $\pi\colon \mathbb{N} \to \mathbb{N}$ is denoted by $\Pi$ and, for such a $\pi$, the meaning of the notation $x_{\pi}$ is clear. Finally, for each subset $I\subseteq \mathbb{N}$ we define $x_I$ to be the sequence given by
 $$
 (x_I)_n = \left\{\begin{array}
                        [c]{cc}
                        x_n & \text{if $n\in I$,}\\
                        0 & \text{otherwise}.
                            \end{array}
                    \right.
$$
Given two sequences $x, y \in \mathbb{R}^{\mathbb{N}}$ we write $x \leq y$ whenever $x_{n} \leq y_{n}$ for every $n$.\\
The next proposition generalizes \cite[Lemma~3 and Corollary]{Gar}.

\begin{proposition}\label{prop closing up}
Let $X$ be a symmetric Banach sequence space.
\begin{enumerate}
\item \label{prop closing up 1} If $\pi\in \Pi$ and $x\in X$, then $x_\pi \in X$ with $\|x_\pi\|_X\leq \|x\|_X$.
\item \label{prop closing up 2} $x\in X$ if and only if $x'\in X$ and, in this case, $\|x\|_X=\|x'\|_X$.
\end{enumerate}
\end{proposition}

\begin{proof}
Bearing Proposition~\ref{aquila} in mind we can split the proof in two cases. We prove first the statement in the case $X\hookrightarrow c_0$ and then we focus on the case $X=\ell_\infty$.

\medskip

\emph{Case $X\hookrightarrow c_0$.} We begin by proving \ref{prop closing up 1}. Let us note first of all that  $x\in c_0$ implies that $x_\pi=(x_{\pi(n)})_{n} \in c_0$ and then, given $\varepsilon>0$, we can
find an infinite set $I_1\subseteq \mathbb{N}$ such that $\sum_{n\in I_1}|x_{\pi(n)}|<\varepsilon$. If we denote $I_2=\mathbb{N}\setminus I_1$ we clearly have
\[
x_\pi=(x_\pi)_{I_1}+(x_\pi)_{I_2} \,.
\]
Our aim is to see that each one of the summands on the right-hand-side belong to $X$. On the one hand $(x_\pi)_{I_1}$ belongs to $\ell_1$, thus to $X$. Moreover \eqref{continuous inclusions} yields
$\|(x_\pi)_{I_1}\|_X\leq\sum_{n\in I_1}|x_{\pi(n)}|<\varepsilon$. In order to see that $(x_\pi)_{I_2}$ also belongs to $X$ we do $I_{2} = \{  i_1,i_2,\dots \}$ and take any permutation $\sigma$ so that
$\sigma(i_{k}) = \pi(i_{k})$ for every $k$. Then $\vert (x_\pi)_{I_2} \vert \leq \vert  x_\sigma \vert$ and, since $x_{\sigma} \in X$ (because $X$ is symmetric), this gives $(x_\pi)_{I_2}\in X$ with
$\|(x_\pi)_{I_2}\|_X\leq \|x_\sigma\|_X=\|x\|_X$. This altogether shows that $x_\pi \in X$ and
$$
\|x_\pi\|_X =\|(x_\pi)_{I_1}+(x_\pi)_{I_2}\|_X\leq \|x\|_X+\varepsilon \,.
$$
Since $\varepsilon>0$ is arbitrary, this proves our claim.\\
Now we prove \ref{prop closing up 2}. On the one hand, it is clear that given $x\in X$ there is some $\pi\in \Pi$ such that $x'=x_\pi$ and, taking the previous case into account, we deduce that $x'\in X$ with $\|x'\|_X\leq \|x\|_X$. We take now $x\in \mathbb{K}^{\mathbb{N}}$ such that $x'\in X$ and we want to see that $x\in X$ with $\|x\|_X\leq \|x'\|_X$. We proceed in a similar way as in \ref{prop closing up 1}, fixing
$\varepsilon>0$ and taking an infinite set $I_1\subseteq \mathbb{N}$ so that $\sum_{n\in I_1}|x'_{n}|<\varepsilon$. We write
\[
I_2=\mathbb{N}\setminus I_1=\{i_1,i_2,\dots\} \text{ and take } J=\{j_1,j_2,\dots\}\subseteq \mathbb{N} \text{ such that } x'_{i_k}=x_{j_k} \text{ for all } k.
\]
Then $x=x_{J} + x_{\mathbb{N} \setminus J}$ and, as before, the goal is to show that each one of these two belong to $X$. First $x_{\mathbb{N}\setminus J} \in X$ with $\|x_{\mathbb{N}\setminus J}\|_X\leq\sum_{n\in I_1}|x'_{n}|<\varepsilon$. Choose a permutation $\sigma$ so that $\sigma(j_{k}) = i_{k}$ for every $k$ and note that $\vert x_{J} \vert \leq \vert  x'_\sigma \vert$, which gives $x_{J}\in X$ with $\|x_{J}\|_X\leq \|x'_\sigma\|_X=\|x'\|_X$. Hence, $\|x\|_X\leq \|x'\|_X+\varepsilon$ for arbitrary $\varepsilon>0$ and this finishes the proof.

\medskip

\emph{Case $X=\ell_\infty$.} It is clear that $x_\pi\in X$ whenever $x\in X$, so we only have to prove the inequality $\|x_\pi\|_X\leq\|x\|_X$. We see first that it suffices to check that the inequality holds in some dense subset $D$. If this is the case, given $x \in X$ we may choose $(x^{k})_{k} \subseteq D$ such that $x^k\xrightarrow[k\to\infty]{\|\cdot\|_X}\, x$. Since $\Vert x_{\pi}^{k_{1}} - x_{\pi}^{k_{2}}  \Vert_{X}
= \Vert (x^{k_{1}} - x^{k_{2}} )_{\pi}  \Vert_{X}
\leq  \Vert x^{k_{1}} - x^{k_{2}}   \Vert_{X}$, the sequence $(x_{\pi}^{k})_{k}$ is Cauchy and
clearly converges coordinate-wise to $x_{\pi}$. This gives $x^k_\pi\xrightarrow[k\to\infty]{\|\cdot\|_X}\, x_\pi$ and the claim follows.\\
We $D$ to be the set of all sequences in $\mathbb{K}^{\mathbb{N}}$ taking finitely many different values (which is dense in $X$) and $x \in D$. The key idea now os to find a sequence $(x_\pi^N)_{N}\subseteq D$ satisfying
\begin{gather}
\|x_\pi^N\|_X \leq \|x\|_X \text{ for every } N\in \mathbb{N} \text{ and } \label{bartalli} \\
\|x_\pi^N\|_X  \xrightarrow[N\to\infty]{} \Vert  x_{\pi} \Vert_{X} . \label{bobet}
\end{gather}
This clearly gives the conclusion. Let us see how can we construct such a sequence. To begin since $x \in D$, the sequence $(x_{\pi(n)})_n$ takes only finitely many values. Then there should be some values, say
$a_1,\dots, a_m$ (that we may assume to be ordered increasingly by the modulus) that are repeated infintely many times. Let us assume that there is at least one value that is repeated only finitely many times
(if this is not the case the argument follows in the same way) and choose $n_0\in \mathbb{N}$ such that $x_{\pi(n_0)}$ is the last time that these values appear in the sequence $x_\pi$ (that is, $x_{\pi(k)} \in \{ a_1,\dots, a_m\}$ for every $k \geq n_{0}$). We fix now $N$ and define
\begin{equation}\label{seiler}
x_\pi^{\sigma_{N}}=(x_{\pi(1)}, \dots, x_{\pi(n_0)},\underbrace{a_1, \dots, a_1}_{\text{$N$ times}}, \underbrace{a_2, \dots, a_2}_{\text{$N$ times}},\dots, \underbrace{a_m, \dots, a_m}_{\text{$N$ times}}, \underbrace{a_1, \dots, a_1}_{\text{$N$ times}},\dots) \,.
\end{equation}
This is a permutation of $x_{\pi}$ and, since $X$ is symmetric, satisfies $\Vert x_\pi^{\sigma_{N}} \Vert_{X} = \Vert x_\pi \Vert_{X}$. We take now the set $\mathbb{N}\setminus \pi(\mathbb{N})$. If it
this is empty, then $\pi$ is a permutation and the result follows from the symmetry of $X$. Assume, then, that it is not empty, write $\mathbb{N}\setminus \pi(\mathbb{N}) = \{n_{1}, n_{2}, \ldots \}$ and define
\begin{gather*}
x_\pi^{N,1}= (x_{\pi(1)}, \dots, x_{\pi(n_0)}, \underbrace{x_{n_1}, a_1, \dots, a_1}_{\text{$N$ times}}, \underbrace{x_{n_2}, a_2, \dots, a_2}_{\text{$N$ times}}, \ldots , \underbrace{x_{n_m}, a_m, \dots, a_m}_{\text{$N$ times}}, \underbrace{x_{n_{m+1}}, a_1, \dots, a_1}_{\text{$N$ times}} \dots),\\
x_\pi^{N,2}= (x_{\pi(1)}, \dots, x_{\pi(n_0)}, \underbrace{a_1,x_{n_1}, \dots, a_1}_{\text{$N$ times}}, \underbrace{ a_2,x_{n_2}, \dots, a_2}_{\text{$N$ times}}, \ldots , \underbrace{ a_m, x_{n_m},\dots, a_m}_{\text{$N$ times}}, \underbrace{a_1,x_{n_{m+1}},  \dots, a_1}_{\text{$N$ times}} \dots),\\
\vdots \\
x_\pi^{N,N}= (x_{\pi(1)}, \dots, x_{\pi(n_0)}, \underbrace{a_1, \dots, a_1,x_{n_1}}_{\text{$N$ times}}, \underbrace{ a_2, \dots, a_2,x_{n_2}}_{\text{$N$ times}}, \ldots , \underbrace{ a_m,\dots, a_m, x_{n_m}}_{\text{$N$ times}}, \underbrace{a_1,  \dots, a_1,x_{n_{m+1}}}_{\text{$N$ times}} \dots) \,.
\end{gather*}
Here we are supposing that $\mathbb{N}\setminus \pi(\mathbb{N})$ is infinite. If it is finite, then we would proceed in this way, placing one $n_{j}$ in each block until we get to the last one. When we run out of
$n_j$'s we go on with blocks consisting only of the corresponding values of $a_{l}$, as in \eqref{seiler}. Now each $x_\pi^{N,j}$ is a permutation of $x$ and, since $X$ is symmetric, we have $\|x_\pi^{N,j}\|_X=\|x\|_X$. Then, if we take
$$
x_\pi^N=\frac{1}{N}\sum_{j=1}^N x_\pi^{N,j},
$$
we have \eqref{bartalli}. On the other hand, for each $n \in \mathbb{N}$ the difference $ (x_\pi^{\sigma_{N}}) - (x_\pi^N)_{n}$ is either $0$ or there is a $j$ and a $k$ so that
\[
 (x_\pi^{\sigma_{N}}) - (x_\pi^N)_{n}  = a_{k} - \frac{(N-1) a_{k} - x_{n_{j}}}{N} = \frac{a_{k} +  x_{n_{j}}}{N} \,.
\]
This gives $\|x_\pi^N - x_\pi^\sigma\|_\infty \xrightarrow[N\to \infty]{}\, 0$ and, since the norm is equivalent, $\|x_\pi^N - x_\pi^\sigma\|_{X} \xrightarrow[N\to \infty]{}\, 0$. This, together with the fact that
$\Vert x_\pi^{\sigma_{N}} \Vert_{X} = \Vert x_\pi \Vert_{X}$ immediately give \eqref{bobet} and complete the proof.

To prove \ref{prop closing up 2} let us first note that, by construction, $x \in \mathbb{K}^{\mathbb{N}}$ is bounded if and only if so is $x'$. So, we only have to check that the equality of norms holds. First of all,
$x\in X$ there is some $\pi\in \Pi$ such that $x'=x_\pi$ and then  \ref{prop closing up 1} gives  $\|x'\|_X\leq\|x\|_X$. To prove the reverse inequality we begin by choosing some $x$ that only takes a finite
number of different values. First of all, if $0$ is the only value that is repeated infinitely many times, then $x'$ is a permutation of $x$ and, consequently, $\|x\|_X=\|x'\|_X$. Suppose, then, that there is some value, say $x_{n_0} \neq 0$ that is repeated infinitely many times in $x$ (hence also in $x'$). Let $J= \{ n \in \mathbb{N} \colon x'_{n} = x_{n_{0}}  \}$ and split this set into two disjoint subsets $J=J_1 \cup J_2$, where the cardinality of $J_2$ equals the number of times that $0$ appears in the sequence $x$. Then we consider
 $$
 y = \left\{\begin{array}
                        [c]{ccc}
                        x_{n_0} & \text{if} & n\in J_1,\\
                        0 & \text{if} & n\in J_2,\\
                        x_n' & \text{if} & n\in \mathbb{N}\setminus J
                            \end{array}
                    \right.
$$
which clearly satisfies $\|y\|_X\leq \|x'\|_X$, since $|y_n|\leq |x_n'|$ for all $n$. On the other hand, $y$ is a permutation of $x$ and, hence,  $\|x\|_X=\|y\|_X\leq \|x'\|_X$. Given an arbitrary $x\in X$ we can take a sequence $(x^k)_k$, where each $x^{k}$ takes only finitely many different values, and such that $x^k\xrightarrow[k\to\infty]{\|\cdot\|_X}\, x$ and $x^k_n=0$ whenever $x_n=0$. This gives $(x^k)'\xrightarrow[k\to\infty]{\|\cdot\|_X}\, x'$ and
$$
\|x\|_X=\lim_{k\to \infty} \|x^k\|_X \leq \lim_{k\to \infty} \|(x^k)'\|_X =\|x'\|_X
$$
which is the desired statement.
\end{proof}

Condition \ref{prop closing up 2} in Proposition~\ref{prop closing up} looks certainly quite similar to Theorem~\ref{simequiv}--\ref{simequiv 2}, and one may wonder if it is also equivalent to the space
being symmetric. The next example, taken from \cite{Gar2} (see also \cite{AlbAnsWal}, where it was generalised) shows that this is not the case.

\begin{example}
Let $\mathcal{O}$ denotes the set of increasing functions form $\mathbb{N}$ to $\mathbb{N}$ and consider the Garling space
\[
g = \Big\{x=(x_n)_n \colon \|x\|_{g}= \sup_{\phi \in \mathcal{O}} \sum_{n=1}^\infty |x_{\phi(n)}| \frac{1}{\sqrt{n}}  \Big\}\,.
\]
It is an easy exercise to show that $x \in g$ if and only if $x' \in g$ (and both have the same norm). However, the space is not symmetric. This was shown in \cite[Section~5]{Gar2} (see also \cite[Theorem~5.10]{AlbAnsWal}); we sketch here the argument. For each fixed $m$ consider
\[
x^{m} = \big(1, \tfrac{1}{\sqrt{2}}, \tfrac{1}{\sqrt{3}}, \ldots , \tfrac{1}{\sqrt{m}}, 0,0 , \ldots\big) \,, \text{ and }
y^{m} = \big(\tfrac{1}{\sqrt{m}}, \ldots, \tfrac{1}{\sqrt{2}}, 1, 0,0 , \ldots\big) \,.
\]
It is straightforward to check that
\begin{equation} \label{gx}
\sup_{m} \Vert x^{m} \Vert_{g} \geq \sup_{m} \sum_{n=1}^{n} \frac{1}{n} = \infty \,.
\end{equation}
For each $\phi \in \mathcal{O}$ we write $m_{\phi} = \max \{ n \colon \phi(n) \leq m  \}$. With this
\[
\Vert y^{m} \Vert_{g} = \sup_{\phi \in \mathcal{O}} \sum_{n=1}^{m_{\phi}} \frac{1}{\sqrt{(m+1-\phi(n))n}} \,.
\]
Note that the only values that play a role in each of these sums are those of $\phi(1), \ldots , \phi(m)$, that are $\leq m$. Then we can find some $k$ and $n_{1} < \cdots < n_{k} \leq m$ so that
\[
\Vert y^{m} \Vert_{g} =  \sum_{i=1}^{k} \frac{1}{\sqrt{(m+1-n_{i})i}} \,.
\]
It can be seen that $n_{i} = m+1-k$ and, then $\Vert y^{m} \Vert_{g}\leq 1 + \pi$ (all the details can be found in \cite[Section~5]{Gar2}). This and \eqref{gx} show, since $x^{m}$ and $y^{m}$ are clearly
a permutation of each other, that $g$ is not symmetric.
\end{example}
%

In the next proposition we state some well-known properties of the decreasing rearrangement of a sequence.
\begin{proposition}\label{properties rearrangement}
\begin{enumerate}
\item \label{properties rearrangement1} If $x,y $ are bounded sequences and $|x| \leq |y|$, then $x^* \leq y^*$.
\item \label{properties rearrangement2} If $x,y $ are bounded sequences satisfying and $x_n=y_{\pi(n)}$ for some injective mapping $\pi\colon \mathbb{N} \to \mathbb{N}$, then $x^*\leq y^*$.
\item\label{prop rearrangement limit} Let $x$ and $x^k$, $k=1,2,\dots$ be bounded sequences.  If $\sup_n |x_n^k-x_n| \xrightarrow[k\to \infty]{} 0$, then $\sup_n |(x^k)^*_n-x^*_n| \xrightarrow[k\to \infty]{} 0$
    
\end{enumerate}
\end{proposition}
\begin{proof}
Item \ref{properties rearrangement1} is clear by definition of the decreasing rearrangement. To prove \ref{properties rearrangement2}, just take $J=\pi(\mathbb{N}) \subseteq \mathbb{N}$ and $y_{J}$.
Note that $x^*=y_J^*$ and, by \ref{properties rearrangement1}, $y_J^*\leq y^*$. Finally, let us prove \ref{prop rearrangement limit}.
Given $\varepsilon >0$, we choose $k_0 \in \mathbb{N}$ such that $|x_n^k-x_n|<\varepsilon$ for all $k\geq k_0$ and all $n\in \mathbb{N}$. In particular $|x_n^k|<|x_n|+\varepsilon$ for every $k\geq k_0$ and all $n\in \mathbb{N}$ which, by \ref{properties rearrangement2}, gives
$$
(x^k)^*\leq ((|x_n|+\varepsilon)_n)^*=(x^*_n+\varepsilon)_n.
$$
Analogously, we can see that $x^*\leq ((x^k)^*_n+\varepsilon)_n$. This shows that  $|(x^k)^*_n-x^*_n|\leq \varepsilon$ for every $k\geq k_0$ and every $n\in \mathbb N$, which completes de proof.
\end{proof}

Note that if $x\in c_{0}$, then $x' \in c_0$ and either $x'_n\neq 0$ for all $n\in \mathbb{N}$ or there is some $n_0\in \mathbb{N}$ so that $x'_n=0$ for every $n\geq n_0$. Consequently, there is a permutation $\sigma$ of $\mathbb{N}$ such that $(\vert x' \vert)_\sigma=(x')^*$. Now, since $x^*=(x')^*$, we have that $x^*=(\vert x' \vert)_\sigma$. This is the last thing we need to proceed with the proof.
%
%
\begin{proof}[Proof of Theorem~\ref{simequiv}]
As we mentioned after the statement of the theorem, we only need to prove that \ref{simequiv 1} implies \ref{simequiv 2}. By Proposition~\ref{aquila} we know that $X\hookrightarrow c_0$ or $X=\ell_\infty$ with some equivalent norm.
Suppose first that $X \hookrightarrow c_0$. If $x\in X$ then by Proposition~\ref{prop closing up} we have $x'\in X$ and, since $x^*=(\vert x' \vert)_\sigma$ for some permutation $\sigma$ and $X$ is symmetric, we deduce that $x^*\in X$ with $\|x^*\|_X = \Vert x' \Vert_{X} =\|x\|_X$. Similarly, if $x^*\in X$ then $x\in X$ with $\|x^*\|_X=\|x\|_X$.

We now prove the theorem in the case $X=\ell_\infty$ with some equivalent norm. It is clear, in this case, that $x\in X$ if and only if $x^*\in X$; hence, we only need to prove the equality $\|x^*\|_X=\|x\|_X$.
As in Proposition~\ref{prop closing up} we take the dense subset $D$ of all the sequences taking only finitely many values. The equality $\|x^*\|_X=\|x\|_X$ is clear for $x\in D$, since $x^*$ is a permutation of $\vert x' \vert$. If we consider any $x\in X$, then there is a sequence $(x^N)_N \subseteq D$ such that  $x^N\xrightarrow[N\to\infty]{\|\cdot\|_X}\, x$.  Since $\| \cdot \|_X$ is equivalent to the sup norm, by Proposition~\ref{properties rearrangement}--\ref{prop rearrangement limit} we have $(x^N)^*\xrightarrow[N\to\infty]{\|\cdot\|_X}\, x^*$ and, hence,
$$
\|x^*\|_X=\lim_{N\to \infty} \|(x_N)^*\|_X = \lim_{N\to \infty} \|x_N\|_X = \|x\|_X.
$$
\end{proof}

\section{Symmetric linear functionals}\label{section sym funct}

We return now to Question~\ref{pregunta} and look at symmetric linear functionals on symmetric Banach sequence spaces. As we pointed out in the Introduction, on $\ell_{1}$ we can define a symmetric linear functional in a rather
easy way. In Proposition~\ref{no symm on linfty} we showed that on $\ell_{\infty}$ there is no such functional, and we asked if there are spaces, other than $\ell_{1}$, where non-zero symmetric linear exist. Our aim  now is to show
that this is indeed the case. We begin by showing that spaces are spanned by the canonical vectors are of no help.

%

\begin{proposition}\label{symm null on separable}
If $X$ is a symmetric Banach sequence space different from $\ell_{1}$ and $\gamma$ is a symmetric linear functional on $X$, then $\gamma(e_{n})=0$ for every $n\in \mathbb{N}$.
\end{proposition}
\begin{proof}
Suppose $\gamma(e_{n})\neq0$ for some (then, by the symmetry of $\gamma$, for every) $n$ and let us see that $X=\ell_1$ (with some equivalent norm). Take  $x=(x_n)_n \in X$ and note that
\begin{multline*}
\sum_{n=1}^N |x_n| = \frac{1}{\gamma(e_1)} \gamma\bigg(\sum_{n=1}^N |x_n| e_n \bigg)
=\Bigg\vert  \frac{1}{\gamma(e_1)} \gamma\bigg(\sum_{n=1}^N |x_n| e_n \bigg) \Bigg\vert\\
\leq \frac{1}{|\gamma(e_1)|} \|\gamma\| \Bigg\| \sum_{n=1}^N |x_n| e_n\Bigg\|_{X}
\leq \frac{\|\gamma\|}{|\gamma(e_1)|} \|x\|_X.
\end{multline*}
for every $N\in \mathbb{N}$. Then $x\in \ell_1$ and $\|x\|_1 \leq \frac{\|\gamma\|}{|\gamma(e_1)|} \|x\|_X$. This shows that $X\hookrightarrow \ell_1$ and, taking \eqref{continuous inclusions} into account, completes the proof.
\end{proof}


As a consequence of the previous proposition, to find a symmetric linear functional we must take a Banach sequence space which is not spanned by the canonical vectors. It can be seen that such a Banach space cannot be separable (more precisely, a Banach sequence space is separable if and only if the canonical vectors form a Schauder basis for it).

The existence of symmetric functionals has been studied by several authors, and is connected with the study of singular traces in the Dixmier sense (see, for instance, \cite{DodPagSemSuk}). The classical examples of Banach spaces admitting symmetric functionals are related to Marcinkiewicz spaces, although there are other examples of symmetric spaces satisfying this property (see \cite{DodPagSedSemSuk}). Our aim is to follow the classical approach in the construction of symmetric functionals defined on Marcinkiewicz spaces using Banach limits, trying to present it in a simple way.\\

In Proposition~\ref{symmetric on marcinkiewicz} we construct a \emph{real} Banach space and a symmetric linear functional on it. Later, we show how transfer this to a \emph{complex} space. So, for a while we work only with real Banach sequence spaces. Given any sequence $x \in \mathbb{R}^{\mathbb{N}}$ we define two sequences $x^{+}, x^{-} \in  \mathbb{R}^{\mathbb{N}}$ by
\[
x_n^+=\sup\{x_n,0\} \text{ and } x_n^-=\sup\{-x_n,0\} \,,
\]
so that $x = x^{+} - x^{-}$. We define the positive cone of a real Banach sequence space $X$ as
\[
X_+=\{x=(x_n)_n \in X \colon x_n\geq0\, \text{ for all } n\in \mathbb{N}\} \,.
\]
The following result, although not being evident, is in some sense implicit when one tries to define a symmetric linear functional. It shows that considering the decreasing rearrangement comes in a natural way when trying to define symmetric functionals .

%
\begin{proposition}
Let $X$ be a real symmetric Banach sequence space and $\gamma\colon X\to \mathbb{R}$ a linear functional.
\begin{enumerate}
\item If $\gamma$ is symmetric then $\gamma(x)=\gamma(x')$ for every $x\in X$. \label{simm closing up}
\item $\gamma$ is symmetric if and only if $\gamma(x)=\gamma(x^*)$ for every $x\in X_+$. \label{simm decreasing rearr}
\end{enumerate}
\end{proposition}
\begin{proof}
By Propositions~\ref{no symm on linfty} and~\ref{aquila} we may assume $X\hookrightarrow c_0$. We begin by proving~\ref{simm closing up}, and choose some $x \in X$. Note first that if $x$ has finitely many
non-zero entries, then both $x$ and $x'$ can be written as a finite linear combinations of $e_{n}$'s and then, by Proposition~\ref{symm null on separable}, $\gamma(x)=\gamma(x')=0$. Suppose now that
there are infinitely many $x_{n}$'s different from $0$. Take some injective $\pi\colon \mathbb{N} \to \mathbb{N}$ so that $x'=x_{\pi}$ and denote $J= \mathbb{N} \setminus \pi(\mathbb{N}) = \{n \colon x_{n} =0 \}$ (observe that this may be finite or infinite). Fix $\varepsilon>0$ and choose some infinite set
$I\subseteq \mathbb{N}$ such that
\[
\sum_{n\in I} |x_{\pi(n)}| <\frac{\varepsilon}{2\|\gamma\|}
\]
(this is possible since $x\in c_0$). We now split $I=I_{1} \cup I_{2}$, where $I_{1}$ is infinite and $I_{2}$ has the same cardinality of $J$ and define a permutation $\sigma\colon \mathbb{N} \to \mathbb{N}$ satisfying
$$
\sigma(n)=\pi(n) \text{ for } n\in \mathbb{N}\setminus I,\quad \sigma(I_1)=J \quad \text{and} \quad \sigma(I_2)=\pi(I) \,.
$$
With this,for each $n \in \mathbb{N}$ we have
\[
x'_{n} - x_{\sigma(n)} =
\begin{cases}
0 & \text{if } n \in \mathbb{N}\setminus I \\
x_{\pi (n)} & \text{ if } n \in I_{1} \\
 x_{\pi (n)} - x_{\pi(i)}& \text{ if } n \in I_{2} \, (\text{ for some } i \in I) \,.
\end{cases}
\]
Then
\[
\|x'-x_{\sigma}\|_{X}
\leq  \sum_{n\in I_{1}} |x_{\pi(n)}| + \sum_{n\in I_{2}} |x_{\pi(n)}|  +\sum_{n\in I} |x_{\pi(n)}| \,.
< \frac{\varepsilon}{\|\gamma\|}
\]
By the symmetry of $\gamma$ we deduce
$$
|\gamma(x')-\gamma(x)|=|\gamma(x')-\gamma(x_\sigma)|=|\gamma(x'-x_\sigma)|\leq \|\gamma\| \|x'-x_\sigma\|<\varepsilon
$$
and, since $\varepsilon>0$ was arbitrary, $\gamma(x)=\gamma(x')$.\\

Let us see now that \ref{simm decreasing rearr} holds. To do that suppose first that $\gamma$ is symmetric and take $x\in X_+$. Since $x\in c_0$ we can find a permutation $\sigma$ so that $x^*=(x')_\sigma$. The symmetry of $\gamma$ and \ref{simm closing up} yield
$$
\gamma(x^*)=\gamma((x')_\sigma)=\gamma(x')=\gamma(x).
$$
Reciprocally, if $x\in X$ and $\sigma$ is any permutation of $\mathbb{N}$ we have
\begin{align*}
\gamma(x_\sigma)&=\gamma(x_\sigma^+)-\gamma(x_\sigma^-)\\
&=\gamma((x_\sigma^+)^*)-\gamma((x_\sigma^-)^*) \quad \text{(by the hypothesis)}\\
&= \gamma((x^+)^*)-\gamma((x^-)^*)\\
&= \gamma(x^+)-\gamma(x^-)=\gamma(x)\quad \text{(by the hypothesis again)} \,.
\end{align*}
This proves that $\gamma$ is symmetric and completes the proof.
\end{proof}

After this result, it is reasonable to expect that the definition of a symmetric linear functional has to involve, in one way or another, the decreasing rearrangement. Just to get an idea of how this could work, way to try to define a
symmetric $\gamma\in \ell_\infty^*$ (which, at this point, we already know that has to be identically zero) would
be the following. The decreasing rearrangement of every $x \in \ell_{\infty}$ is a convergent sequence, then we could define
\begin{equation}\label{first try symmetric}
\gamma(x)=\lim_{n\to \infty} x_n^*.
\end{equation}
This is obviously symmetric, but is far from being linear, since $\gamma(x)=\gamma(-x)$ for every $x$. There is a way to overcome this problem by defining the functional just on the positive cone and then extending it in a convenient way.

\begin{remark} \label{seguidilla}
Let $X$ be a symmetric Banach space and  $\varphi : X_{+} \to \mathbb{R}$ a continuous function satisfying
\begin{equation}\label{gandula}
\varphi(x+y) = \varphi(x)+\varphi(y) \text{ for every } x,y \in X_+
\end{equation}
and
\begin{equation} \label{cercavila}
\varphi(\lambda x)=\lambda \varphi(x) \text{ for every } x\in X_+ \text{ and } \lambda \geq 0 \,.
\end{equation}
Then the mapping $\gamma \colon X \to \mathbb{R}$ defined by
\begin{equation}\label{extension of phi to X}
\gamma(x)=\varphi(x^+)-\varphi(x^-),
\end{equation}
is linear and continuous. Moreover, if $\varphi(x)=\varphi(x_\sigma)$ for every permutation $\sigma$ and every $x\in X_+$, then the extension $\gamma$ is symmetric.
\end{remark}
Unfortunately (or not), we cannot define $\varphi$ on $(\ell_\infty)_+$ as in \eqref{first try symmetric} and then extend it to $\ell_\infty$. Although such a mapping clearly satisfies \eqref{cercavila}, taking $x=(1,0,1,0,\dots)$ and $y=(0,1,0,1,\dots)$ in $(\ell_\infty)_+$ shows that \eqref{gandula} does not hold.

We can solve this problem by dealing with Marcinkiewicz sequence spaces. We recall the definition of these spaces. Let $\Psi \colon \mathbb{N} \to \mathbb{R}$ be such that $\Psi(1)=1$,  $\Psi(n) \nearrow \infty$ and $\frac{\Psi(n)}{n}\to 0$. Then, the Marcinkiewicz sequence space associated to $\Psi$ is defined by
$$
m_\Psi=\bigg\{x=(x_n)_n\in \mathbb{R}^{\mathbb{N}} \colon
\|x\|_{m_\Psi}=\sup_{n \in \mathbb{N}} \frac{\sum_{k=1}^nx_k^*}{\Psi(n)}<\infty\bigg\}.
$$
It is well known that $(m_\Psi, \|\cdot\|_{m_\Psi})$ is a symmetric Banach sequence space. Indeed, Proposition~\ref{properties rearrangement}~\ref{properties rearrangement1} and the definition of the norm show that
for $x \in \mathbb{K}^{\mathbb{N}}$ and $y \in m_\Psi$ such that  $\vert x_{n} \vert \leq \vert y_{n} \vert$ for every $n$, we have $x \in m_\Psi$ and $\Vert x \Vert_{m_\Psi} \leq \Vert y \Vert_{m_\Psi}$. Symmetry is clear, and linearity and completeness are left as an exercise. Also, since $\frac{\Psi(n)}{n}\to 0$ we see that $(1,1,1,\dots)$ does not belong to $m_\Psi$. By Proposition~\ref{aquila}, $m_\Psi$ is contained in $c_0$.
  Now, trying to repeat the previous argument, given $x \in (m_{\Psi})_+$  we would like to consider
$$
\varphi(x)=\lim_{n\to \infty}  \frac{\sum_{k=1}^nx_k^*}{\Psi(n)}.
$$
This is not possible in this case, since the sequence $\Big( \frac{\sum_{k=1}^nx_k^*}{\Psi(n)}\Big)_n$ is not necessarily convergent (see Example~\ref{ej-sin-lim}). The solution in this case is to take instead Banach limits,
an extension of the classical idea of limit to bounded sequences (all needed details are given in the Appendix~\ref{appendix Banach limits}). So, given a Banach limit $L$ we may define $\varphi : (m_\Psi)_+ \to \mathbb{R}$ as
\begin{equation}\label{idea construction gamma symmetric}
\varphi(x)=L\bigg( \bigg(\frac{\sum_{k=1}^n x_k^*}{\Psi(n)}\bigg)_n\bigg) \,.
\end{equation}
This satisfies $\varphi (x) = \varphi (x_{\sigma})$ for every $x$ and every permutation $\sigma$. Then, if we are able to extend the mapping to $m_{\Psi}$ as in \eqref{extension of phi to X} we will get the desired symmetric
linear functional. This mapping $\varphi$ clearly satisfies \eqref{cercavila}, so the challenge now is to find conditions on $\Psi$ so that \eqref{gandula} holds. Everything relies on the following lemma. For each bounded
$x=(x_{n})_{n}$ we write $s_n(x)=\sum_{k=1}^n x_k^*$.

\begin{lemma}\label{lemma marcinkiewicz}
If $x,y\in (m_\Psi)_{+}$ then
\begin{equation*}
s_n(x+y)\leq s_n(x)+s_n(y)\leq s_{2n}(x+y).
\end{equation*}
\end{lemma}
\begin{proof}
Note first that, since $x+y\in (c_0)_{+}$ and we can find a permutation $\sigma$ so that $(x+y)^*=(x+y)_{\sigma} = x_{\sigma} + y_{\sigma}$. Then
\[
s_n(x+y) =\sum_{k=1}^n x_{\sigma(k)}+\sum_{k=1}^n y_{\sigma(k)}\leq s_n(x)+s_n(y).
\]
To check the second inequality take two permutations $\mu, \eta$ o that $x^{*} = x_\mu$, $y^{*} = y_{\eta}$ and fix $n$. On one hand, if $\mu(\{ 1, \ldots , n  \}) = \eta (\{ 1, \ldots , n  \})$, choose any permutation $\sigma$
so that $\sigma (\{ 1, \ldots , n  \}) = \mu (\{ 1, \ldots , n  \})$. Then
\[
s_n(x)+s_n(y)=\sum_{k=1}^n x_{\mu(k)}+ \sum_{k=1}^n y_{\eta(k)}
= \sum_{k=1}^{n} x_{\sigma(k)}+  \sum_{k=1}^{n} y_{\sigma(k)}
\leq \sum_{k=1}^{2n}(x+y)_{\sigma(k)} \leq s_{2n}(x+y).
\]
On the other hand, if  $\mu(\{ 1, \ldots , n  \}) \neq \eta (\{ 1, \ldots , n  \})$, we may choose $\{j_1, \dots, j_m\}\subseteq \{1, \dots, n\}$ such that $\{\eta(j_1), \dots, \eta(j_m)\}\cap \{\mu(1), \dots, \mu(n)\}=\emptyset$ and take any permutation $\sigma$ so that
$$
\sigma(1)=\mu(1), \dots, \sigma(n)=\mu(n), \sigma(n+1)=\eta(j_1), \dots, \sigma(n+m)=\eta(j_m).
$$
Then
$$
s_n(x)+s_n(y)=\sum_{k=1}^n x_{\mu(k)}+ \sum_{k=1}^n y_{\eta(k)}\leq \sum_{k=1}^{n+m} x_{\sigma(k)}+  \sum_{k=1}^{n+m} y_{\sigma(k)} \leq \sum_{k=1}^{2n}(x+y)_{\sigma(k)} \leq s_{2n}(x+y).
$$
\end{proof}

As an immediate consequence we have
\begin{equation}\label{inequality marcinkiewicz}
\frac{s_n(x+y)}{\Psi(n)}\leq \frac{s_n(x)+s_n(y)}{\Psi(n)}\leq \frac{s_{2n}(x+y)}{\Psi(2n)}\frac{\Psi(2n)}{\Psi(n)}
\end{equation}
for every $n$. Hence, what we have to ask $\Psi$ is
\begin{equation}\label{bochini}
\lim_{n\to\infty}\frac{\Psi(2n)}{\Psi(n)}=1,
\end{equation}
(this is satisfied, for example, by $\Psi(n)=\log(n+1)$).

\begin{proposition}\label{symmetric on marcinkiewicz}
If $\Psi $ satisfies~\eqref{bochini}, then there exist symmetric linear functionals on $m_\Psi$.
\end{proposition}
\begin{proof}
We choose a Banach limit  $L$ satisfying
  \begin{equation}\label{Banach limit eq}
L((x_n)_n)=L(x_1, x_1, x_2, x_2, x_3, x_3, \dots)
\end{equation}
(see Lemma~\ref{lemma 1 Banach limit}) and define $\varphi \colon (m_{\Psi})_{+} \to \mathbb{R}$ as in \eqref{idea construction gamma symmetric}. As we already pointed out, it is only left to check that this satisfies
\eqref{gandula}. Since $L(x)\geq 0$ for every $x\geq 0$, \eqref{inequality marcinkiewicz} gives
$$
L\bigg(\bigg(\frac{s_n(x+y)}{\Psi(n)}\bigg)_n\bigg)
\leq L \bigg(\bigg(\frac{s_n(x)+s_n(y)}{\Psi(n)}\bigg)_n\bigg)
\leq L \bigg( \bigg(\frac{s_{2n}(x+y)}{\Psi(2n)}\frac{\Psi(2n)}{\Psi(n)} \bigg)_n\bigg).
$$
With this, the linearity of $L$, condition \eqref{bochini} and Lemma~\ref{lemma 2 Banach limit}, we get
$$
\varphi(x+y) \leq \varphi(x) + \varphi(y)
\leq L \bigg(\bigg( \frac{s_{2n}(x+y)}{\Psi(2n)} \bigg)_n \bigg) \,.
$$
Thus the proof will be completed once we see that $L\Big(\Big(\frac{s_{2n}(x+y)}{\Psi(2n)}\Big)_n\Big)
= L \Big( \Big( \frac{s_{n}(x+y)}{\Psi(n)} \Big)_n\Big)$. Denote $c_n=\frac{s_{n}(x+y)}{\Psi(n)}$ and observe that, by \eqref{Banach limit eq},
$$
L((c_{2n})_n)=L(c_2, c_4, c_6,\dots)=L(c_2, c_2, c_4, c_4, c_6, c_6,\dots).
$$
It is not difficult to see that $c_{2n}-c_{2n-1}\xrightarrow[n\to \infty]{}\,0$ and, hence,
$$
(c_2, c_2, c_4, c_4, c_6, c_6,\dots) - (c_1,c_2,c_3,c_4,\dots) \in c_0.
$$
Since $L(x)=0$ for every $x\in c_0$, we deduce
$$
L(c_2, c_2, c_4, c_4, c_6, c_6,\dots)=L(c_1,c_2,c_3,c_4,\dots)
$$
and then $L((c_{2n})_n)=L((c_{n})_n)$, which is the desired statement.
\end{proof}

\begin{remark}
Proposition~\ref{symmetric on marcinkiewicz} delivers a real Banach sequence space and a symmetric real linear functional. This can be easily transferred to the complex case. Given $\Psi$ satisfying \eqref{bochini} we consider
the complex Marcinkiewicz sequence space
$$
m_\Psi=\bigg\{x=(x_n)_n\in \mathbb{C}^{\mathbb{N}} \colon
\|x\|_{m_\Psi}=\sup_{n \in \mathbb{N}} \frac{\sum_{k=1}^nx_k^*}{\Psi(n)}<\infty\bigg\}.
$$
Then we have a symmetric real-linear functional $\gamma \colon m_{\Psi} \cap \mathbb{R}^{\mathbb{N}} \to \mathbb{R}$. Letting
\[
\tilde{\gamma} (x) = \gamma \big( (\re x_{n})_{n} \big) + i \gamma \big( (\im x_{n})_{n} \big)
\]
defines a symmetric, complex linear functional $\tilde{\gamma} \colon m_{\Psi} \to \mathbb{C}$.
\end{remark}

\appendix
\section{Banach limits}\label{appendix Banach limits}

\epigraph{ I am sick of symmetry.}{\textit{The Phantom of Liberty}\\ \textsc{Luis Buñuel} }

In the construction of a symmetric linear functional done in Proposition~\ref{symmetric on marcinkiewicz} we used a Banach limit. This is a generalization of the classical limit to bounded (not necessarily convergent) sequences.
We recall now the definition of such a functional, and prove two basic properties that were used in Section~\ref{section sym funct}.\\

We denote by $c$ the subspace of $\ell_{\infty}$ consisting of convergent sequences. Letting  $L(x) = \lim_{n\to\infty} x_n$ defines a linear functional $L : c \to \mathbb{K}$ that satisfies  $\|L\|=1$,
$L((x_n)_{n \geq 1})=L((x_n)_{n\geq 2})$ (is shift invariant), and $L(x)\geq 0$whenever $x\geq 0$. Then, by the Hahn-Banach theorem, this can be extended to a linear functional $L\colon \ell_\infty\to \mathbb{K}$, that moreover
(see \cite[Chapter~III]{Co90}) preserves all these properties:
\begin{itemize}
\item $\|L \colon \ell_{\infty} \to \mathbb{K}\|=1$;
\item if $x\in c$ then $L(x)=\lim_{n\to\infty}x_n$;
\item if $x\in \ell_\infty$ is such that $x_n\geq 0$ for every $n\in \mathbb{N}$, then $L(x)\geq 0$;
\item $L((x_n)_{n \geq 1})=L((x_n)_{n\geq 2})$ for every $(x_n)_n\in \ell_\infty$.
\end{itemize}
Every linear functional $L\colon \ell_\infty\to \mathbb{K}$ satisfying these properties is called a \emph{Banach limit}.


\begin{lemma}\label{lemma 1 Banach limit}
There exist a Banach limit $L$ satisfying
$$
L((x_n)_n)=L(x_1, x_1, x_2, x_2, x_3, x_3, \dots)
$$
\end{lemma}
for every $x \in \ell_\infty$.
\begin{proof}
Take the dilation operator $D_2\colon \ell_\infty \to \ell_\infty$ defined by $$D_2(x_1, x_2, x_3,\dots)=(x_1, x_1, x_2, x_2, x_3, x_3, \dots)$$
and consider a Banach limit $\phi $. For each $n\in \mathbb{N}$ define $\phi_n\in \ell_\infty^*$ by
$$
\phi_{n}=\frac{1}{n+1}\sum_{k=0}^n \phi\circ D_2^k,
$$
which is again a Banach limit. Now, by the Banach-Alaoglu theorem, there exists a subnet $(\phi_{n_i})_{i\in I}$
weakly$^{*}$ convergent to some $L \in (\ell_{\infty})^{*}$. It is easy to check that $L$ is again a Banach limit.  Our goal, then, is to prove that $L\circ D_2=L$. For this purpose, just note that
\begin{multline*}
|\phi_{n_i}\circ D_2(x)-\phi_{n_i}(x)|
= \left| \frac{1}{n_i+1}\sum_{k=0}^{n_i} \phi\circ D_2^{k+1}(x) - \frac{1}{n_i+1}\sum_{k=0}^{n_i} \phi\circ D_2^k(x)\right|\\
= \frac{1}{n_i + 1} |\phi\circ D_2^{n_i+1}(x) - \phi(x)|
\leq \frac{2\|x\|_\infty}{n_i + 1} .
\end{multline*}
Taking the limit on $i$ we obtain the desired result.
\end{proof}

\begin{lemma}\label{lemma 2 Banach limit}
Let $L$ be a Banach limit.
If $a_n \xrightarrow[n\to \infty]{}\, a$ and $(x_n)_n\in \ell_\infty$ then $$L((a_n x_n)_n)=a\cdot L((x_n)_n).$$
\end{lemma}
\begin{proof}
Noting that $(a_n-a)x_n \xrightarrow[n\to \infty]{}\, 0$ we have $L(((a_n-a)x_n)_n)=0$. Then, by linearity, $L((a_n x_n)_n)=a\cdot L((x_n)_n)$.
\end{proof}

We finish this note by giving an example that shows that we can find sequences $x$ in some Marcinkiewicz space
for which $\Big(\frac{\sum_{k=1}^nx_k^*}{\Psi(n)}\Big)_n$ does not converge. This is why we need
Banach limits in order to construct symmetric linear functionals.

\begin{example}\label{ej-sin-lim}
Choose $\Psi(n) = \log (n+1)$ and consider the corresponding space $m_{\Psi}$. Fix $N_{1} =1$ and take $0 < c_{1} < 1$ so that
\[
\frac{N_{1}c_{1}}{\log (N_{1}+1)} =1 \,.
\]
Let $N_{2} = (N_{1} + 1)^{4} - (N_{1} + 1) \in \mathbb{N}$ and observe that in this case,
\[
\frac{N_{1}c_{1}}{\log (N_{1}+N_{2}+1)}
= \frac{N_{1}c_{1}}{\log (N_{1}+1)} \times \frac{\log (N_{1}+1)}{\log (N_{1}+N_{2}+1)}  = \frac{1}{4} \,.
\]
Taking now $c_{2} \leq \min \big\{ c_{1} , \frac{\log (N_{1}+N_{2}+1)}{4N_{2}}  \big\}$ gives
\[
\frac{N_{1}c_{1} + N_{2}c_{2}}{\log (N_{1}+N_{2}+1)} \leq \frac{1}{2} \,.
\]
Now we have
$$
\frac{N_1 c_1+N_2 c_2 +n c_2}{\log (N_1+N_2+n+1)}=\frac{N_1+N_2 +n}{\log(N_1+N_2+n+1)}\times\frac{N_1 c_1+N_2 c_2 +n c_2}{N_1+N_2+n} \xrightarrow[n\to\infty]{}\,\infty
$$
and we can take $N_3$, the first natural number such that
$$
\frac{N_1 c_1+N_2 c_2 +N_3 c_2}{\Psi(N_1+N_2+N_3)}>1.
$$
Since $g(t)=\frac{N_1 c_1+N_2 c_2 +N_3 t}{\Psi(N_1+N_2+N_3)}$ is continuous and $g(0)<1/2<1<g(c_2)$, we can choose $0<c_3<c_2$ such that $g(c_3)=1$. Proceeding in this way we define two sequences $c_1\geq c_2 \geq c_3 \geq \cdots >0$ and $N_1, N_2, N_3,\ldots \in \mathbb{N}$ so that
\[
\frac{N_1 c_1+\cdots+N_k c_k}{\log(N_1+\cdots+N_k+1)}= 1 \text{ if $k$ is odd}
\]
and
\[
\frac{N_1 c_1+\cdots+N_k c_k}{\log(N_1+\cdots+N_k+1)} < \frac{1}{2}  \text{ if $k$ is even.}
\]
We now consider the sequence
$$
x=(\underbrace{c_1,\ldots,c_1}_{\text{$N_1$ times}}, \underbrace{c_2,\ldots, c_2}_{\text{$N_2$ times}}, \underbrace{c_3,\ldots, c_3}_{\text{$N_3$ times}}, \ldots) \,,
$$
for which  $\Big( \frac{\sum_{k=1}^nx_k^*}{\log(n+1)} \Big)_n$ does not converge. Let us see that $x \in m_{\Psi}$. If $n=N_1+\cdots+N_k$ for some $k$ then $\frac{\sum_{k=1}^nx_k^*}{\log(n+1)}$ is either $1$ or $1/2$. We suppose then that $n=N_1+\cdots+N_k+m$  with $1 \leq m < N_{k+1}$ and distinguish two cases. If $k$ is even,  $N_{k+1}$ is by construction the smallest natural number such that
$$
\frac{N_1 c_1+\cdots +N_k c_k+N_{k+1}c_k}{\log(N_1+\cdots+N_{k+1} +1)}>1 \,.
$$
This and the facts that $c_{k+1}\leq c_k$ and $m < N_{k+1}$ give
$$
\frac{\sum_{k=1}^nx_k^*}{\log(n+1)}
=\frac{N_1 c_1+\cdots+N_k c_k+m c_{k+1}}{\log(N_1+\cdots+N_k+m +1)} \leq 1.
$$
Finally, if $k$ is odd,
\begin{multline*}
\frac{\sum_{k=1}^nx_k^*}{\log(n+1)}
= \frac{N_1 c_1+\cdots+N_k c_k}{\log(N_1+\cdots+N_k+m +1)}
+ \frac{mc_{k+1}}{\log(N_1+\cdots+N_k+m +1)} \\
= \frac{N_1 c_1+\cdots+N_k c_k}{\log(N_1+\cdots+N_k+m +1)}
\Big( 1 + c_{k+1} \frac{m}{N_1 c_1+\cdots+N_k c_k}  \Big) \\
\leq \frac{N_1 c_1+\cdots+N_k c_k}{\log(N_1+\cdots+N_k+1)}
\Big( 1 + c_{k+1} \frac{m}{N_1 c_1+\cdots+N_k c_k}  \Big) \,.
\end{multline*}
Taking into account that, by construction
\[
\frac{N_1 c_1+\cdots+N_k c_k}{\log(N_1+\cdots+N_k+1)} = 1 , \,
c_{k+1} \leq \frac{\log(N_1+\cdots+N_{k+1}+1)}{4N_{k+1}} \text{ and }
\frac{N_1 c_1+\cdots+N_k c_k}{\log(N_1+\cdots+N_{k+1}+1)} = \frac{1}{4} \,,
\]
we finally get $\frac{\sum_{k=1}^nx_k^*}{\log(n+1)} \leq 2$. This shows that indeed $x \in m_{\Psi}$.
\end{example}

\bibliographystyle{abbrv}
\bibliography{Biblio_CaMaSe}

\begin{thebibliography}{1}

\bibitem{AlbAnsWal}
F.~{Albiac}, J.~L. {Ansorena}, and B.~{Wallis}.
\newblock {Garling sequence spaces.}
\newblock {\em {J. Lond. Math. Soc., II. Ser.}}, 98(1):204--222, 2018.

\bibitem{Co90}
J.~B. Conway.
\newblock {\em A course in functional analysis}, volume~96 of {\em Graduate
  Texts in Mathematics}.
\newblock Springer-Verlag, New York, second edition, 1990.

\bibitem{DodPagSedSemSuk}
P.~G. Dodds, B.~de~Pagter, A.~A. Sedaev, E.~M. Semenov, and F.~A. Sukochev.
\newblock Singular symmetric functionals and {B}anach limits with additional
  invariance properties.
\newblock {\em Izv. Ross. Akad. Nauk Ser. Mat.}, 67(6):111--136, 2003.

\bibitem{DodPagSemSuk}
P.~G. Dodds, B.~de~Pagter, E.~M. Semenov, and F.~A. Sukochev.
\newblock Symmetric functionals and singular traces.
\newblock {\em Positivity}, 2(1):47--75, 1998.

\bibitem{Gar}
D.~J.~H. {Garling}.
\newblock {On symmetric sequence spaces.}
\newblock {\em {Proc. Lond. Math. Soc. (3)}}, 16:85--106, 1966.

\bibitem{Gar2}
D.~J.~H. {Garling}.
\newblock {Symmetric bases of locally convex spaces.}
\newblock {\em {Stud. Math.}}, 30:163--181, 1968.

\bibitem{LiTz77}
J.~Lindenstrauss and L.~Tzafriri.
\newblock {\em Classical {B}anach spaces. {I}}.
\newblock Springer-Verlag, Berlin-New York, 1977.
\newblock Sequence spaces, Ergebnisse der Mathematik und ihrer Grenzgebiete,
  Vol. 92.

\end{thebibliography}

\end{document}